\documentclass[reqno]{amsart}

\usepackage[absolute,overlay]{textpos}
\usepackage{eso-pic}
\usepackage[a4paper,margin=1.25in]{geometry}
\usepackage{mathrsfs}
\usepackage{amsfonts,  amsmath}
\usepackage{graphicx}
\usepackage{enumerate,  enumitem,  amscd}
\usepackage{amsthm,  amssymb,  epsfig, amsmath}  
\usepackage[hidelinks]{hyperref}
\usepackage{mathtools}
\usepackage{tikz}
\usepackage{tikz-cd}
\usetikzlibrary{positioning, arrows,automata,arrows.meta}
\usepackage{quiver}
\usepackage{subcaption}
 
\tikzset{
  vx/.style  = {circle, draw=black, fill=blue!20, line width=0.6pt,
                minimum size=8pt, inner sep=0pt},
  ed/.style  = {draw=black!60, line width=0.5pt},
  lbl/.style = {font=\tiny},
  ttl/.style = {font=\small}
}

\newcommand{\pres}[3]{\textnormal{#1} \langle #2 \mid #3 \rangle}

\newcommand{\ZZ}{\mathcal{Z}}

\newcommand{\Z}{\mathbb{Z}}

\newcommand{\N}{\mathbb{N}}
\newcommand{\ab}{\mathrm{ab}}

\newcommand{\cF}{\mathcal{F}}
\newcommand{\cC}{\mathcal{C}}
\newcommand{\cG}{\mathcal{G}}
\newcommand{\cX}{\mathcal{X}}

\newcommand{\cT}{\mathcal{T}}
\newcommand{\cL}{\mathcal{L}}

\DeclareMathOperator{\Stab}{Stab}

\DeclareMathOperator{\im}{im}

\DeclareMathOperator{\Fix}{Fix}

\newtheorem{theorem}{Theorem} 
\newtheorem*{theorem*}{Theorem} 
\numberwithin{theorem}{section}
\newtheorem{lemma}[theorem]{Lemma}     

\newtheorem*{corollary*}{Corollary}
\newtheorem{proposition}[theorem]{Proposition}
\newtheorem*{proposition*}{Proposition}
\newtheorem{question}{Question}
\newtheorem*{question*}{Question}
\numberwithin{question}{section}

\theoremstyle{definition}

\numberwithin{example}{section}

\newtheorem*{remark*}{Remark}

\begin{document}

\title[Free-by-cyclic groups are not relatively profinitely rigid]{Free-by-cyclic groups are not relatively \\ profinitely rigid}
\author{Carl-Fredrik Nyberg-Brodda}
\address{June E Huh Center for Mathematical Challenges, Korea Institute for Advanced Study (KIAS), Seoul 02455, Korea}
\email{cfnb@kias.re.kr}

\thanks{The author is supported by KIAS Individual Grant HP094701 at Korea Institute for Advanced Study, and by the Mid-Career Researcher Program (RS-2023-00278510) through the National Research Foundation of the Government of Korea.}

\date{\today}

\keywords{}
\subjclass[2020]{}

\begin{abstract} 
For $m=5$ and for every $m \geq 7$ we construct families of pairwise non-isomorphic free-by-cyclic groups isomorphic to $F_{2m} \rtimes \Z$ having isomorphic profinite completions. This answers a question of Bridson \& Reid from 2015. 
\end{abstract}

\maketitle

\noindent The question of to which extent a finitely generated, residually finite group can be determined up to isomorphism by its finite quotients -- the question of \textit{profinite rigidity} -- is a classical area of group theory that has recently seen a great deal of activity, with many breakthrough results. Deep work by Bridson, McReynolds, Reid \& Spitler \cite{BMRS} has given the first examples of profinitely rigid hyperbolic $3$-manifold groups and a toolbox for constructing more, with many other profinite rigidity results following in its wake, see e.g.\ \cite{BMRS21,Bridson2022}.

Rather than consider profinite rigidity relative to all finitely generated residually finite groups, one can restrict to consider \textit{relative} profinite rigidity. Thus, for a class $\cC$ of finitely generated residually finite groups, a group $G \in \cC$ is \textit{profinitely rigid relative to $\cC$} if any group $H \in \cC$ with the same finite quotients as $G$ is necessarily isomorphic to $G$. In recent years, the question of profinite rigidity relative to the class of \textit{(f.g.\ free)-by-cyclic groups} has been studied. This was initiated by Bridson \& Reid \cite{Bridson2020} (preprint published in 2015) who in 2015 asked whether free-by-cyclic groups are determined by their finite quotients. Many articles have targeted this question, including recent work by Bridson, Reid \& Wilton \cite{Bridson2017}, Hughes \& Kudlinska \cite{Hughes2025}, and Bridson \& Piwek \cite{Bridson2025}. Many of these results point towards rather than away from free-by-cyclic groups being relatively profinitely rigid. 

In quite the opposite direction, we will in this article show that free-by-cyclic groups are \textit{not} relatively profinitely rigid, thus answering Bridson \& Reid's question negatively. Specifically, we will construct a family of groups $G_{p,q}$ for $p, q \in \Z$ with the following properties:

\begin{theorem*}
For every $p \geq 2$ and $q, r \in (\Z / p\Z)^\times$, we have:
\begin{enumerate}[label=\normalfont(\Roman*)]
\item $\widehat{G_{p,q}} \cong \widehat{G_{p,r}}$,
\item $G_{p,q} \cong F_{2p} \rtimes \Z$ is free-by-cyclic, and
\item $G_{p,q} \cong G_{p,r}$ if and only if $q \equiv \pm r \pmod{p}$.
\end{enumerate}
\end{theorem*}
In particular, as soon as $\frac{\phi(p)}{2}>1$, we get two non-isomorphic free-by-cyclic groups with the same profinite completion. The smallest such $p$ is $5$, so that $G_{5,1} \not\cong G_{5,2}$, but they have the same finite quotients, and are both free-by-cyclic groups of the form $F_{10} \rtimes \Z$. As soon as $p \geq 7$ we have $\phi(p) \geq 4$, and hence we get the result announced in the abstract. Furthermore, the size of the relative profinite genus of $G_{p,q}$ grows unboundedly in $p$, i.e.\ for every $N \in \N$ there exists a free-by-cyclic group with relative profinite genus of cardinality larger than $N$. 

A peculiarity of the proof of our Theorem is that the proofs of its three parts (I), (II), and (III) are all more or less independent of one another. We will give an outline of the components of the proof of the Theorem in \S\ref{Subsec:crossroads} below, including a brief discussion of the order(s) in which the sections of this article may be read.

\subsection*{Acknowledgements} I wish to thank Martin Bridson and Alan Reid for encouragement and helpful discussions. Several steps in the proof were simplified with the use of GPT-5.6 (Sol) accessed via ChatGPT Pro. All words in this article were written by it (human) author.

\section{Background and crossroads}\label{Sec:background}

\subsection{Profinite rigidity} \label{Subsec:profinite-rigidity-back}

We give some background on profinite rigidity; for (much) more, we refer the reader to Reid's 2018 ICM address on the topic \cite{Reid2018}. Let $G$ be a finitely generated residually finite group, let $\cF(G)$ be the set of all finite quotients of $G$, given the structure of an inverse system in the natural way. The profinite completion $\widehat{G}$ of $G$ is then the inverse limit of this system. We say that $G$ is (absolutely) \textit{profinitely rigid} if it is uniquely distinguished by its set of finite quotients $\cF(G)$ amongst all other finitely generated, residually finite groups. This condition $\cF(G_1) = \cF(G_2)$ is well-known to be equivalent to $\widehat{G}_1 \cong \widehat{G}_2$ \cite{Dixon1982, Nikolov2007}. The \textit{genus} $\cG(G)$ of $G$ is the set of isomorphism classes of finitely generated, residually finite groups with the same profinite completion as $G$. 

On the other hand, if we fix a set $\cC$ of isomorphism types of finitely generated groups (e.g.\ the set of f.g.\ free groups or the set of one-relator groups), then we say that $G \in \cC$ is \textit{profinitely rigid relative to $\cC$} if $\widehat{G} \cong \widehat{H}$ implies $G \cong H$ for every $H \in \cC$. This leads to defining $\cG(\cC, G)$, the \textit{genus relative to $\cC$}, as $\cG(G) \cap \cC$. Thus profinite rigidity relative to $\cC$ asks whether or not $\cG(\cC, \cG) = \{ G \}$. The relative genus was first introduced by Grunewald \& Zalesskii \cite{Grunewald2011}. In general, the (absolute) genus of $G$ can be quite different from its relative genera. For example, all free groups are profinitely rigid relative to the set of free groups, since the finite quotients of a rank-$n$ free group $F_n$ are precisely the $n$-generated finite groups. On the other hand, the \textit{absolute} profinite rigidity (i.e.\ profinite rigidity relative to the class of all finitely generated groups) of free groups remains a famous open problem \cite[Question~12]{Noskov1979}.

\subsection{Free-by-cyclic groups}\label{Subsec:free-by-cyclic}

Recall that a group $G$ is said to be \textit{(f.g.\ free)-by-cyclic} if it fits into a short exact sequence $1 \to F_n \to G \to \Z \to 1$, where $F_n$ is a finitely generated free group; since all such extensions split, this is equivalent to saying $G \cong F_n \rtimes \Z$ for some action of $\Z$. We will use the term \textit{free-by-cyclic} throughout this article as an abbreviation for this definition; recent results of Linton \cite{Linton2026} (namely that finitely generated \{free-by-cyclic\} groups are precisely the finitely generated subgroups of \{f.g.\ free\}-by-cyclic groups) show that this is only somewhat abusive. Baumslag \cite{Baumslag1971} showed that all free-by-cyclic groups are residually finite, so the question of profinite rigidity is a very natural one for this class. In the sequel, when we speak of the relative profinite rigidity of free-by-cyclic groups, we will always mean profinite rigidity relative to the class of all free-by-cyclic groups.

Bridson, Reid \& Wilton \cite{Bridson2017} proved, among other results, that the groups $F_2 \rtimes \Z$ are relatively profinitely rigid.  Recently, there have been several further results in this direction. Hughes \& Kudlinska \cite{Hughes2025} proved that many classes of free-by-cyclic groups are \textit{almost} relatively profinitely rigid (i.e.\ that they have finite relative genus) -- the largest such class being the class of irreducible free-by-cyclic groups with first Betti number $b_1(G)=1$. Bridson \& Piwek \cite{Bridson2025} recently proved that all free-by-cyclic groups with non-trivial center are relatively profinitely rigid, and Andrew, Hillen, Lyman \& Pfaff \cite{Andrew2025} recently gave an example of a \textit{hyperbolic} relatively profinitely rigid free-by-cyclic group.

\subsection{The groups $G_{p,q}$} We will now define the free-by-cyclic groups $G_{p,q}$ that constitute the core of the main Theorem. First, fix $p \geq 2$, let $q \in \Z$, and let $B_p = C_p \ast C_p \ast \Z = \langle a, b, c \mid a^p = b^p = 1 \rangle$. We add a central element $h$, and define a central extension $A_{p,q}$ of $B_p$ by 
\begin{equation}\label{Eq:Apq-definition}
A_{p,q} = \pres{}{a,b,c,h}{[h,a] = [h,b] = [h,c] = 1, \: a^p h^{q} = 1, \: b^p h^{-q} = 1}.
\end{equation}
We will intentionally keep the same names $a, b, c$ for the generators of $A_{p,q}$ as for $B_p$, as it will simplify the notation used throughout the article greatly. To ensure no confusion occurs, we will always make the ambient group clear.

Next, consider the element $abc$ of $A_{p,q}$. Both $c$ and $abc$ have infinite order in $A_{p,q}$, since $c$ and $abc$ both project to infinite-order elements in the abelianization of $A_{p,q}$. Thus, we can form an HNN extension of $A_{p,q}$ with the associated (cyclic) subgroups $\langle c \rangle$ and $\langle abc \rangle$, and a new stable letter $s$. We call this extension $G_{p,q}$, so that
\begin{equation}\label{Eq:Gpq-definition}
G_{p,q} = \pres{}{A_{p,q}, s}{s^{-1} c s = abc}
\end{equation}
These are the groups that constitute the groups in our main Theorem. 

\subsection{Crossroads}\label{Subsec:crossroads}

In this article, there are three things to be proved, being the three statements (I), (II), and (III) of the main Theorem. In presenting our proofs of these three parts, we place the reader at a crossroads: the proof of each of the three pieces is independent of every other, and they can thus be read in any order, giving the reader the choice between one of six possible articles. It is the view of the author that the choice of order (I), (II), (III) is the most natural, since it follows the approximate order of difficulty: the first part (I) is not difficult, and uses only elementary combinatorial group theory via generators and relations; the second part (II) uses mostly routine techniques, but requires a touch of Bass--Serre theory; while the third part (III) requires a fairly heavy dose of Bass--Serre theory, following Serre \cite{Serre}. The outline of the article with this choice of order, together with a sketch of each proof, is as follows: 
\begin{enumerate}[label=(\Roman*)]
\item In \S\ref{Sec:same-profinite}, we prove that the groups $G_{p,q}$ and $G_{p,r}$ have the same finite quotients. This is achieved by elementary number theoretic considerations, allowing any surjective homomorphism $G_{p,r} \to Q$ onto some finite group $Q$ to induce a homomorphism $G_{p,q} \to Q$, which then becomes surjective if one assumes that $\gcd(p, q) = \gcd(p, r) = 1$. 
\item In \S\ref{Sec:groups-are-freebycyclic}, we prove that $G_{p,q}$ is free-by-cyclic. To do this, we obtain a homomorphism $\eta \colon G_{p,q} \to \Z$, and then study its kernel in two steps. First, by using some Bass--Serre theory, we prove that the kernel is isomorphic to $N \ast \Z$, where $N = \ker(\eta) \cap A_{p,q}$. Since $A_{p,q}$ is a central extension of $B_p$, this intersection is then quite easily reduced to a problem about a certain finite index subgroup of $B_p$. The group $B_p$ is a virtually free group, and so an Euler characteristic argument then lets us conclude that $N \cong F_{2p-1}$. Thus the kernel becomes isomorphic to $F_{2p}$, so that $G_{p,q} \cong F_{2p} \rtimes \Z$. 
\item In \S\ref{Subsec:rarely-isomorphic}, the most technically involved of the three parts, we prove that $G_{p,q} \cong G_{p,r}$ if and only if $q \equiv \pm r \pmod{p}$. The converse implication is not difficult (and not necessary for our main Theorem), so the forward direction constitutes the entire difficulty. On a high level, we will show that any isomorphism $G_{p,q} \cong G_{p,r}$ gives rise to a very rigid isomorphism $\Phi \colon G_{p,q} \to G_{p,r}$, and that this rigidity is enough to recover the numbers $q$ and $r$ (up to sign mod $p$). In detail, we start the section with a detailed analysis of the centralizers in $G_{p,q}$ and $A_{p,q}$ via Bass--Serre theory. We then assume that $G_{p,q} \cong G_{p,r}$, and using the centralizer information, we show that there must exist some isomorphism $\Phi \colon G_{p,q} \to G_{p,r}$ which maps $\Phi(A_{p,q}) = A_{p,r}$ (Lemma~\ref{Lem:iso-descends-to-A}). Thus we can reconstruct the base group of the HNN extension; we next show that any such $\Phi$ must, up to conjugacy, also either fix the two associated subgroups $\langle c \rangle$ and $\langle abc \rangle$ or interchange them. That is then sufficient information to, via a simple elementary number-theoretic argument (and information about conjugate subgroups in the virtually free group $B_p$), show that we necessarily have $q \equiv \pm r \pmod{p}$. 
\end{enumerate}

We end this section with two open problems. First, as noted in the introduction, our main Theorem shows that taking $p=5$ already gives a pair of non-isomorphic free-by-cyclic groups $F_{10} \rtimes \Z$ with the same profinite completion. We do not believe that this rank 10 is by any means optimal, but as mentioned in \S\ref{Subsec:free-by-cyclic} it is known that all groups $F_2 \rtimes \Z$ are relatively profinitely rigid. This raises the following natural question:

\begin{question}
What is the greatest $2 \leq m \leq 9$ such that all free-by-cyclic groups $F_m \rtimes \Z$ are relatively profinitely rigid? 
\end{question}

Finally, as mentioned in the introduction, although the lower bounds for the relative genera we construct in this article can be arbitrarily large, the bounds are nevertheless always finite. Thus we ask:

\begin{question}\label{Quest:finite}
Is there some free-by-cyclic group $G$ such that the relative genus of $G$ is infinite? In particular, is the relative genus of some $G_{p,q}$ infinite? 
\end{question}

Our groups are reducible with first Betti number $3$, so we do not have any concrete tools to know whether their relative genus is finite or not (cf.\ the results by Hughes \& Kudlinska \cite{Hughes2025}). We do not believe that the methods in this article are particularly suited for this question. 

\clearpage

\section{The groups have isomorphic profinite completions}\label{Sec:same-profinite}

\noindent In this section, we will prove that $\widehat{G_{p,q}} \cong \widehat{G_{p,r}}$ whenever $q, r \in (\Z / p\Z)^\times$, which is the easiest part of our argument; it requires little more than elementary number theory and manipulating some defining relations. As noted in the introduction, to do this it suffices to prove that $G_{p,q}$ and $G_{p,r}$ have precisely the same set of finite quotients. 

\begin{proposition}\label{Prop:same-quotients}
Let $p \geq 2$. For all $q, r \in (\Z / p\Z)^\times$, the groups $G_{p,q}$ and $G_{p,r}$ have the same finite quotients, i.e.\ $\widehat{G_{p,q}} \cong \widehat{G_{p,r}}$. 
\end{proposition}
\begin{proof}
Let $Q$ be any finite group such that $G_{p, r} \twoheadrightarrow Q$. We will show that $G_{p,q}$ also surjects $Q$. First, fix some $t \in \Z$ such that 
\[
qt \equiv r \pmod{p} \quad \text{and} \quad \gcd(t, |Q|) = 1.
\]
Such $t$ is easy to construct. Indeed, we require $t \equiv rq^{-1} \pmod{p}$, and for every prime $\ell \mid |Q|$ with $\ell$ not dividing $p$ we can add the condition $t \equiv 1 \pmod{\ell}$. Then by the Chinese Remainder Theorem, there exists a $t$ satisfying these conditions simultaneously, which thus satisfies both $qt \equiv r \pmod{p}$ and $\gcd(t, |Q|) =1$. Let then $\gamma = \frac{r-qt}{p}$, which by construction satisfies $\gamma \in \Z$. 

Let the generators of $G_{p,q}$ be $a_q, b_q, c_q, h_q, s_q$, and let $\widetilde{a}_r, \widetilde{b}_r, \widetilde{c}_r, \widetilde{h}_r, \widetilde{s}_r$ denote the images under the given surjection of the generators of $G_{p,r}$ in $Q$. Then we define a map $\Phi \colon G_{p,q} \to Q$ by 
\begin{align}\label{Eq:im-of-h_q-generators}
a_q \mapsto \widetilde{a}_r \widetilde{h}_r^\gamma, \quad b_q \mapsto \widetilde{b}_r & \widetilde{h}_r^{-\gamma}, \quad c_q \mapsto \widetilde{c}_r, \quad h_q \mapsto \widetilde{h}_r^{t}, \quad s_q \mapsto \widetilde{s}_r.
\end{align} 
Let us verify that this is a homomorphism. First, all centrality relations regarding $h_{q}$ from $A_{p,q}$ are certainly satisfied, since $\Phi(h_q) = \widetilde{h}_r^{t}$ commutes with $\widetilde{a}_r, \widetilde{b}_r$, and $\widetilde{c}_r$. Next, the two relations $a_q^p h_q^q = 1$ and $b_q^p h_q^{-q}=1$ get mapped to, respectively
\begin{align*}
\Phi(a_q^p h_q^q) &= (\widetilde{a}_r \widetilde{h}_r^\gamma)^p \widetilde{h}_r^{tq} = \widetilde{a}_r^{p} \widetilde{h}_r^{\gamma + qt} = \widetilde{h}_r^{-r + p\gamma + qt} =  \widetilde{h}_r^{0} = 1, \\ 
\Phi(b_q^p h_q^{-q}) &= (\widetilde{b}_r \widetilde{h}_r^{-\gamma})^p \widetilde{h}_r^{-tq} = \widetilde{b}_r^{p} \widetilde{h}_r^{-p\gamma - tq} = \widetilde{h}_r^{r-p\gamma - qt} = \widetilde{h}_r^0 = 1,
\end{align*}
where we have made ample use of the centrality of $\widetilde{h}_r$ together with the facts that $\widetilde{a}_r^p = \widetilde{h}_r^{-r}$ and $\widetilde{b}_r^p=\widetilde{h}_r^r$, both of which follow from the fact that the map onto $Q$ is a homomorphism. Hence $\Phi$ respects these relations. Thus it remains only to check that the HNN relation $s_q^{-1}c_qs_q = a_q b_q c_q$ is satisfied. This is also straightforward: indeed, $\Phi(c_q) = \widetilde{c}_r$, and $\Phi(a_q b_q c_q) = (\widetilde{a}_r \widetilde{h}_r^\gamma)(\widetilde{b}_r \widetilde{h}_r^{-\gamma}) \widetilde{c}_r = \widetilde{a}_r \widetilde{b}_r \widetilde{c}_r$, using centrality of $\widetilde{h}_r$. Hence $s_q^{-1}c_qs_q = a_q b_q c_q$ is mapped to $\widetilde{s}_r^{-1}\widetilde{c}_r\widetilde{s}_r = \widetilde{a}_r \widetilde{b}_r \widetilde{c}_r$, so  the relation $s_q^{-1} c_q s_q = a_qb_qc_q$ is also preserved by $\Phi$. Thus $\Phi$ is a homomorphism $G_{p,q} \to Q$. 

It remains to verify that $\Phi$ is surjective. Since $\gcd(t, |Q|)=1$ by assumption, there exists $k \in \Z$ with $tk \equiv 1 \pmod{|Q|}$, and then $\widetilde{h}_r = (\widetilde{h}_r^t)^k = \Phi(h_q^k)$, so that $\widetilde{h}_r \in \im(\Phi)$. But then it is easy to recover all other generators, too, by using \eqref{Eq:im-of-h_q-generators}. Hence all the images of the generators of $G_{p,r}$ lie in the image of $\Phi$, and since the map $G_{p,r} \to Q$ is surjective, hence so too is $\Phi$. Thus $G_{p,q}$ surjects $Q$, and since $Q$ was arbitrary, all finite images of $G_{p,r}$ are also finite images of $G_{p,q}$. Interchanging $q$ and $r$ in the above proof thus proves that the two groups have exactly the same finite quotients. 
\end{proof}

As an example of the parameters in Proposition~\ref{Prop:same-quotients}, if we take $p=5$ and $q=2, r=4$, and $|Q| = 30$, then we may choose $t = 7$ and $\gamma = -2$. 

\section{The groups are free-by-cyclic}\label{Sec:groups-are-freebycyclic}

\noindent In this section, we will prove that the group $G_{p,q}$ as defined in \eqref{Eq:Gpq-definition} is (f.g.\ free)-by-cyclic for all coprime $p, q \geq 2$. This requires some Bass--Serre theory and exploiting the fact that $G_{p,q}$ is an HNN extension, but is otherwise rather straightforward. Specifically, we will show the following proposition.

\begin{proposition}\label{Prop:they-are-free-by-cyclic}
Let $p,q \geq 2$ be coprime integers. Then the group $G_{p,q}$ is isomorphic to $F_{2p} \rtimes \Z$, and is in particular (f.g.\ free)-by-cyclic.
\end{proposition}
\begin{proof}
We define a map $\eta \colon G_{p,q} \to \Z$ by 
\begin{equation}\label{Eq:eta}
\eta(a) = -q, \quad \eta(b) = q, \quad \eta(c) = 1, \quad \eta(h) = p, \quad \eta(s) = 0.
\end{equation}
This is a homomorphism, since $\eta(a^p h^q) = \eta(b^p h^{-q}) = 0$, the centrality relations are obviously respected, and $\eta(abc) = -q+q+1 = 1 = \eta(c)$, so that the relation $s^{-1}cs = abc$ is preserved. Since $\eta(c)=1$, it is surjective, and since $c \in A_{p,q}$, the restriction $\eta |_{A_{p,q}} \colon A_{p,q} \to \Z$ is also surjective. 

Let $K = \ker(\eta)$. We will show that $K$ is free of rank $2p$. Let $N = K \cap A_{p,q}$, which is the kernel of $\eta|_{A_{p,q}}$. We first show that $K \cong N \ast \Z$. Recall that $G_{p,q}$ is an HNN extension of $A_{p,q}$ with associated cyclic subgroups $\langle c \rangle$ and $\langle abc \rangle$ and stable letter $s$. Let $\cT$ be the Bass--Serre tree of this HNN extension. Then $K$ acts on $\cT$, and the quotient graph $K \setminus \cT$ has vertices corresponding to double cosets $K \setminus G_{p,q} / A_{p,q}$. However, because $\eta(A_{p,q}) = \Z$, it follows that every element $g \in G_{p,q}$ can be written as $g = k a$ with $k \in K$ and $a\in A_{p,q}$, so that $G_{p,q} = K A_{p,q}$. Hence there is only one double coset $K \setminus G_{p,q} / A_{p,q}$, so that $K \setminus \cT$ has only one vertex. The edges of $K \setminus \cT$ correspond to the double cosets of the edge group of the HNN extension, i.e.\ $K \setminus G_{p,q} / \langle c \rangle$. Since $\eta(c) = 1$, there is exactly one such double coset: indeed, for any $g_1, g_2 \in G_{p,q}$, let $m = \eta(g_2) - \eta(g_1)$. Then $\eta(g_2 (g_1 c^m)^{-1}) = 0$, so $g_2 = kg_1 c^m$ for some $k \in K$. Hence $K g_1 \langle c \rangle = K g_2 \langle c \rangle$, so there is only one double coset $K \setminus G_{p,q} / \langle c \rangle$. Hence $K \setminus \cT$ has exactly one edge.

Finally, the stabilizer in $K$ of the standard edge is $K \cap \langle c \rangle$, but since $\eta(c^m) = m$ is zero if and only if $m=0$, it follows that $K \cap \langle c \rangle = 1$. Since $K$ is normal in $G_{p,q}$, every conjugate edge stabilizer is also trivial. Hence we have, in total, shown that $K \setminus \cT$ is a graph of groups with one vertex and one loop edge; the vertex group is $N = K \cap A_{p,q}$, and the edge group is trivial. Thus the fundamental group of this graph of groups $K \setminus \cT$ is isomorphic to the free product $N \ast \Z$, and hence $K \cong N \ast \Z$ by the fundamental theorem of Bass--Serre theory. 

We now show that $N = K \cap A_{p,q}$ is free of rank $2p-1$, where $K = \ker(\eta)$. First, note that $\langle h \rangle$ is a central subgroup in $A_{p,q}$, and we have the quotient
\[
B_p := A_{p,q} / \langle h \rangle \cong \pres{}{a,b,c}{a^p = b^p = 1} \cong C_p \ast C_p \ast \Z.
\]
We thus obtain a map $\eta_p \colon B_p \to \Z / p \Z$, by lifting each element of $B_p$ to a pre-image of $A_{p,q}$, applying $\eta$, and taking reduction mod $p$. Since $\eta(h) = p$, the $\eta$-value of any pre-image differs only by a multiple of $p$, so $\eta_p$ is a well-defined map; it maps $\eta_p(a) = -q$, $\eta_p(b) = q$, and $\eta_p(c) = 1$. We claim that the projection $A_{p,q} \to B_p$ above restricts to an isomorphism $N \cong \ker(\eta_p)$. To see that it is injective, note that $N \cap \langle h \rangle = 1$, since $\eta(h^m) = pm$ vanishes if and only if $m=0$. Hence the quotient by $\langle h \rangle$ leaves $N$ unaffected, and $N$ injects into $\ker(\eta_p)$. On the other hand, to show surjectivity, let $\overline{g} \in \ker(\eta_p)$, and let $g \in A_{p,q}$ be any lift. Since $\eta_p(\overline{g}) = 0$, the integer $\eta(g)$ is divisible by $p$. Hence $gh^{-\eta(g)/p}$ maps under $\eta$ to $\eta(g) - \eta(g) = 0$, so $gh^{-\eta(g)/p} \in N$, and it is a lift of $\overline{g}$, since it maps to $\overline{g}$ upon killing $h$. Thus $N \cong \ker(\eta_p)$. 

Finally, note that $B_p$ is a virtually free group, so computing $\ker(\eta_p)$ is now routine. Indeed, $\eta_p$ is injective on both of the free factors $C_p$ in $B_p$, since $\eta_p(a) = -q$ and $\eta_p(b) = q$ are both invertible modulo $p$ by assumption. Thus $\ker(\eta_p)$ meets every conjugate of either finite factor trivially. Hence, by the Kurosh subgroup theorem, $\ker(\eta_p)$, and hence also $N \cong \ker(\eta_p)$, is a free group. Its rank follows from an easy (rational) Euler characteristic argument. First, we have
\[
\chi(B_p) = \chi(C_p) + \chi(C_p) + \chi(\Z) - 2 = \frac{1}{p} + \frac{1}{p} + 0 - 2 = \frac{2}{p} - 2.
\]
Now $\eta_p$ is surjective onto $\Z / p\Z$, so $\ker(\eta_p)$ has index $p$ in $B_p$. Since the Euler characteristic of a finite index subgroup is that of the overgroup multiplied by the index of the subgroup, we find $\chi(N) = \chi(\ker(\eta_p)) = p \chi(B_p) = 2 - 2p$. We showed above that $N$ is free, and hence it must be free of rank $2p-1$. Since $K \cong N \ast \Z$, as shown above, we have $K \cong F_{2p}$. Since $G_{p,q}$ fits into the short exact sequence $1 \to K \to G_{p,q} \xrightarrow{\eta} \Z \to 1$, this completes the proof of Proposition~\ref{Prop:they-are-free-by-cyclic}.
\end{proof}

\section{The groups are rarely isomorphic}\label{Subsec:rarely-isomorphic}

\noindent In this section our goal is to prove the following result, which is both the most technical as well as the most important part of the article.

\begin{proposition}\label{Prop:rarely-isomorphic}
Fix $p \geq 2$ and let $q, r \in \Z$. Then 
\[
G_{p,q} \cong G_{p,r} \quad \iff \quad r \equiv \pm q \pmod{p}.
\]
\end{proposition}

To prove this proposition, we will require much information about the centralizers in $G_{p,q}$, and in particular we will make significant use of the decomposition of $G_{p,q}$ as an HNN extension. 

We recall some notation. Throughout this section, unless otherwise specified, we fix $p \geq 2$ and $q \in \Z$ and let $G = G_{p,q}$. We retain the names of the generators $a, b, c, h, s$ from the presentation \eqref{Eq:Gpq-definition} of $G$. Similarly, $G$ is an HNN extension of a group $A := \langle a, b, c, h \rangle$ identifying $\langle c \rangle$ with $\langle abc \rangle$ with the stable letter $s$. The center of $A$ is $\langle h \rangle$, and $B := A / \langle h \rangle$ is isomorphic to $C_p \ast C_p \ast \Z$. We will sometimes use the names $a, b, c$ also for their images in $B$. We denote the centralizer of an element $x$ in a group $G$ by $C_G(x)$ throughout, and likewise the normalizer by $N_G(x)$. 

We record some easy facts about centralizers in $B$. Since $B = \langle a, b, c \rangle \cong C_p \ast C_p \ast \Z$, the centralizer of every non-trivial element of $B$ is cyclic. Furthermore, the two subgroups $\langle c \rangle$ and $\langle abc \rangle$ of $B$ are both maximal infinite cyclic and malnormal. The maximality and malnormality of the free factor $\langle c \rangle$ are both immediate, and the map $a \mapsto a, b \mapsto b, c \mapsto abc$ is an automorphism of $B$, so $\langle abc \rangle$ is also malnormal. On the other hand, $\langle c \rangle$ and $\langle abc \rangle$ are not conjugate to one another, as can be seen by projecting to the abelianization.

\subsection{Centralizers in $A$}\label{Subsec:centralizers-in-A}

We first prove a lemma about some centralizers in the base group $A$. Recall that $A$ is a central extension of the virtually free group $B$, so the following statements are all more or less routine. 

\begin{lemma}\label{Lem:Centralizers-in-A}
The following statements all hold about centralizers in $A$. 
\begin{enumerate}[label=(\roman*)]
\item If $x \in A \setminus \langle h \rangle$, then $C_A(x)$ is abelian. 
\item For every $m \neq 0$, we have 
\begin{equation}\label{Eq:centralizers-of-u^m}
C_A(c^m) = \langle h, c \rangle, \quad \text{and} \quad C_A((abc)^m) = \langle h, abc \rangle.
\end{equation}
\item We have $N_A(\langle c \rangle) = C_A(c)$ and $N_A(\langle abc \rangle) = C_A(abc)$. 
\end{enumerate}
\end{lemma}
\begin{proof}
First, we prove (i).  As mentioned above, the centralizer of any non-trivial element in $B$ is cyclic, so if $x \in A \setminus \langle h \rangle$, then its image in $B$ is non-trivial, and hence $C_A(x)$ is contained in the inverse image of some cyclic subgroup of $B$. That inverse image is generated by one lift together with the central element $h$, and is hence abelian.

Next, we prove (ii). Fix some $m \neq 0$. The image of $c^m$ in $B$ is again $c^m$, and as noted above the subgroup $\langle c \rangle$ is maximal cyclic and malnormal in $B$, so that $C_B(c^m) = \langle c \rangle$. Taking the inverse image in $A$ gives $C_A(c^m) = \langle h, c \rangle$, as desired. Similarly, since $c \mapsto abc$ extends to an automorphism of $B$, and taking inverse images, we get $C_A((abc)^m) = \langle h, abc \rangle$. 

Finally, we prove (iii). The inclusion(s) $\supseteq$ are obvious, so suppose $x \in N_A(\langle c \rangle)$. Then modulo $\langle h \rangle$, the element $x$ normalizes $\langle c \rangle$ in $B$, so malnormality gives that $x \in \langle c \rangle$ in $B$. Hence $x \in \langle h, c \rangle$ in $A$. Since $\langle h, c \rangle$ centralizes $c$, and since $x$ was arbitrary, we thus get $N_A(\langle c \rangle) \subseteq C_A(c)$ and we thus get the equality. The proof for $\langle abc \rangle$ is identical.
\end{proof}

\subsection{Centralizers in $G$}\label{Subsec:centralizers-in-G}

To describe centralizers in $G$, we require its description as an HNN extension of $A$, as recalled in the beginning of \S\ref{Sec:groups-are-freebycyclic}. In particular, let $\cT$ be the Bass--Serre tree of $G$, and recall that $s$ is the stable letter conjugating the two subgroups $\langle c \rangle$ and $\langle abc \rangle$. The vertices of $\cT$ are cosets of $A$. Let $x_0 \in V(\cT)$ be the vertex represented by $A$, let $x_1$ be the vertex represented by $sA$, and let $y_{0,1}$ be the standard edge connecting $x_0$ to $x_1$. Finally, we will make use of the homomorphism $\eta \colon G \to \Z$ defined by \eqref{Eq:eta}, though we shall not require any of the properties of it proved in \S\ref{Sec:groups-are-freebycyclic} (except, of course, that it is well-defined). We note that $\eta(c) = \eta(abc) = 1$.

\begin{lemma}\label{Lem:Centralizers-in-G}
The following statements all hold about centralizers in $G$.
\begin{enumerate}[label=(\roman*)]
\item For every $m \neq 0$, we have $C_G(h^m) = A$. 
\item For every $m \neq 0$, we have $C_G(c^m) = \langle h, c \rangle \ast_{\langle c \rangle} \langle shs^{-1}, c \rangle \cong F_2 \times \Z$.
\item If $g \in G$ acts hyperbolically on $\cT$, then $C_G(g)$ is abelian. 
\end{enumerate}
\end{lemma}
\begin{proof}
We begin with (i). Fix an $m \neq 0$. The element $h^m$ fixes the vertex $x_0$, since $h^m \in A$. We claim that $h^m$ fixes no edge incident at $x_0$. The stabilizers of the edges incident at $x_0$ are precisely the $A$-conjugates of $\langle c \rangle$ and $\langle abc \rangle$. If $h^m$ fixed such an edge, then for some $\alpha \in A$ and non-zero $k \in \Z$ we would have either $h^m = \alpha c^k \alpha^{-1}$ or else $h^{m} = \alpha (abc)^k \alpha^{-1}$. Passing to $B = A / \langle h \rangle$, this would give $c^k=1$ or $(abc)^k=1$ in $B$, both of which are impossible, since $c$ and $abc$ both have infinite order. Since the fixed-point set of $h^m$ acting on $\cT$ is a subtree, it follows that $h^m$ fixes exactly $\{ x_0 \}$, and every element commuting with $h^m$ must hence also fix $x_0$. Hence $C_G(h^m) \leq \Stab_G(x_0) = A$. Since $h$ is central in $A$, we get the reverse inclusion.

\

Next, we prove (ii). Again, fix some $m \neq 0$. Let $X = \Fix_\cT(c^m)$, i.e.\ $\cX$ is the fixed subtree of the action of $c^m$ on $\cT$. Recall that $y_{0,1}$ is the edge from $x_0$ to $x_1$ in $\cT$. Since $c^m \in \Stab_G(y_{0,1}) = \langle c \rangle$, the tree $\cX$ is non-empty and contains $y_{0,1}$. The group $C_G(c^m)$ also preserves $\cX$, since $c^m(zx) = z(c^mx) = zx$ for every $z \in C_G(c^m)$ and $x \in \cX$. We show that this action of $C_G(c^m)$ on $\cX$ is transitive on the edges of $\cX$. Let $y$ be any edge of $\cX$. Now $G$ acts transitively on the edges of $\cT$, since $\cT$ is the Bass--Serre tree of the HNN extension, so we may write $y = g \cdot y_{0,1}$ for some $g \in G$. Because $c^m$ fixes $y$, we have $c^m \in \Stab_G(y) = g \langle c \rangle g^{-1}$. Hence $g^{-1} c^m g = c^k$ for some $k \in \Z$. Applying the map $\eta \colon G \to \Z$ to both sides, we get $m=k$, and hence $g^{-1}c^m g = c^m$. Thus $g \in C_G(c^m)$, so $y = g \cdot y_{0,1}$ lies in the $C_G(c^m)$-orbit of $y_{0,1}$. Since the edge $y$ was arbitrary, this shows the claimed transitivity. 

We will now show that the endpoints $x_0$ and $x_1$ of $y_{0,1}$ lie in distinct $C_G(c^m)$-orbits. Suppose for contradiction that for some $z \in C_G(c^m)$ we have $z \cdot x_0 = x_1$. Since $x_0 = A$ and $x_1 = sA$, we have $z = sa_0$ for some $a_0 \in A$. Since $z$ centralizes $c^m$, we get $c^m = z^{-1}c^m z = a_0^{-1}s^{-1}c^m sa_0 = a_0^{-1} (abc)^m a_0$. Passing into the quotient $B$, this would imply that $c^m$ is conjugate to $(abc)^m$ in $B$. But this is clearly not true, as then their centralizers $\langle c \rangle$ resp.\ $\langle abc \rangle$ would be conjugate, which is false as observed earlier by projecting to the abelianization. 

Thus the quotient graph $C_G(c^m) \setminus \cX$ is a segment, with endpoints $x_0$ and $x_1$, and one edge $y_{0,1}$. The stabilizer in $C_G(c^m)$ of $x_0$ is $C_G(c^m) \cap A = C_A(c^m) = \langle h, c \rangle$ by Lemma~\ref{Lem:Centralizers-in-A}(ii), while the stabilizer in $C_G(c^m)$ of $x_1$ is 
\[
C_G(c^m) \cap sAs^{-1} = sC_A(s^{-1}c^m s)s^{-1} = sC_A((abc)^m)s^{-1} =  s \langle h, abc \rangle s^{-1} = \langle shs^{-1}, c \rangle
\]
again by Lemma~\ref{Lem:Centralizers-in-A}(ii). The stabilizer of the edge $y_{0,1}$ is just $C_G(c^m) \cap \langle c \rangle = \langle c \rangle$. Thus, since $C_G(c^m)$ is, by the fundamental theorem of Bass--Serre theory, isomorphic to the fundamental group of the segment just constructed, we have the amalgam decomposition
\begin{equation}\label{Eq:amalgam-for-cgu}
C_G(c^m) = \langle h, c \rangle \ast_{\langle c \rangle} \langle shs^{-1}, c \rangle \cong \pres{}{c, \alpha, \beta}{[c, \alpha] = [c, \beta] = 1}.
\end{equation}
This is obviously isomorphic to $F_2 \times \Z$, i.e.\ $h$ and $shs^{-1}$ generate a free group, so we have (ii).

\

Finally, we prove (iii). Let $g \in G$ be hyperbolic, and let $\cL$ be its axis. Every element of $C_G(g)$ preserves $\cL$, so we obtain a homomorphism $C_G(g) \to \Z$. The kernel of this map fixes $\cL$ pointwise and is therefore contained in an edge stabilizer, and is hence cyclic. If the kernel is trivial, then $C_G(g)$ is cyclic and we are done. Otherwise, let $z$ be a non-trivial element in the necessarily infinite cyclic kernel. For every $x \in C_G(g)$, we have that $xzx^{-1}$ is in the kernel, and since the kernel is infinite cyclic, $xzx^{-1} = z^{\pm 1}$. Since $z$ lies in the conjugate of an edge stabilizer, there are $y \in G$ and a non-zero $k \in \Z$ such that $z = y c^k y^{-1}$. But then $\eta(z) = \eta(c^k) = k \neq 0$, so we conclude $xzx^{-1} = z$. Thus the kernel is central in $C_G(g)$, so $C_G(g)$ is a central extension of a cyclic group and hence abelian. 
\end{proof}

This completes the results about the centralizers in $G$ that we will require. 

\subsection{Reconstructing the base group from an isomorphism}\label{Subsec:reconstructing-A}

We will now show that any isomorphism $G_{p,q} \xrightarrow{\sim} G_{p,r}$ must, up to an inner automorphism, descend to an isomorphism $A_{p,q} \xrightarrow{\sim} A_{p,r}$. Once this is done, we will show that any \textit{such} isomorphism in turn forces enough control over the images of $c$ and $abc$ to show that $q \equiv \pm r \pmod{p}$. 

For ease of reading, we will throughout this section use the notation $C_{p,q}$ resp.\ $C_{p,r}$ for centralizers in $G_{p,q}$ resp.\ $G_{p,r}$, rather than use the unwieldy $C_{G_{p,q}}$ and $C_{G_{p,r}}$. We will also re-use the notation of subscripting the generators of $G_{p,q}$ resp.\ $G_{p,r}$ by $q$ resp.\ $r$, so that e.g.\ $G_{p,q}$ is generated by $a_q, b_q, c_q, h_q, s_q$. 

\begin{lemma}\label{Lem:iso-descends-to-A}
Let $p \geq 2$ and $q, r \in \Z$ be such that there exists an isomorphism $\Phi \colon G_{p,q} \xrightarrow{\sim} G_{p,r}$. Then there exists $g \in G_{p,r}$ such that $\Phi(A_{p,q}) = g A_{p,r} g^{-1}$. 
\end{lemma}
\begin{proof}
By Lemma~\ref{Lem:Centralizers-in-G}(i), we have $A_{p,q} = C_{p,q}(h_q)$, and hence $\Phi(C_{p,q}(h_q)) = C_{p,r}(\Phi(h_q))$. Let $z = \Phi(h_q)$. The group $\Phi(A_{p,q})$ is isomorphic to $A_{p,q}$, so it is non-abelian; furthermore, the quotient modulo its center is isomorphic to $C_p \ast C_p \ast \Z$, which we shall presently use. Now the element $z \in G_{p,r}$ is either hyperbolic or elliptic. It cannot be hyperbolic, since by Lemma~\ref{Lem:Centralizers-in-G}(iii), any hyperbolic element has abelian centralizer. Hence $z$ is elliptic, so it fixes some vertex. We claim $z$ cannot fix an edge: if it did, then it would lie in a conjugate of an edge stabilizer, and hence would be conjugate to $c_r^m$ or $(a_rb_rc_r)^m$ for some $m \neq 0$. By Lemma~\ref{Lem:Centralizers-in-G}(ii), the centralizer of $c_r^m$ (and hence also of $(a_rb_rc_r)^m$, since it is conjugate to $c_r^m$) is isomorphic to $F_2 \times \Z$, which would thus also be true of the centralizer of $z$. But $F_2 \times \Z$ modulo its center is $F_2$, which is not isomorphic to $C_p \ast C_p \ast \Z$, a contradiction. Hence $z$ fixes no edge, and hence it must fix exactly one vertex of the Bass--Serre tree of $G_{p,r}$. Its centralizer also fixes that vertex, so for some $g \in G_{p,r}$ we have $\Phi(A_{p,q}) = C_{p,r}(z) \leq g A_{p,r} g^{-1}$. 

Let now $y = g^{-1}zg$. Then $y \in A_{p,r}$, and $C_{p,r}(y) = g^{-1} C_{p,r}(z) g$, so $C_{p,r}(y)$ is non-abelian. Since $y$ fixes no edge, it only fixes $x_0$, and hence any element centralizing it must also only fix it; thus $C_{p,r}(y) \leq A_{p,r}$, so $C_{p,r}(y) = C_{A_{p,r}}(y)$. By Lemma~\ref{Lem:Centralizers-in-A}(i), the only elements with non-abelian centralizers in $A_{p,r}$ lie in $\langle h_r \rangle$. Hence $y \in \langle h_r \rangle$. Since $y \neq 1$ (since $z = \Phi(h_q) \neq 1$), it follows from Lemma~\ref{Lem:Centralizers-in-G}(i) that $C_{p,r}(y) = A_{p,r}$. But $C_{p,r}(y) = g^{-1} C_{p,r}(z) g$, so conjugating by $g$ we get $\Phi(A_{p,q}) = C_{p,r}(z) = g A_{p,r} g^{-1}$. The lemma is proved.  
\end{proof}

Consequently, if $G_{p,q} \cong G_{p,r}$, we may compose any such isomorphism by some inner automorphism to conclude that there exists an isomorphism $\Phi \colon G_{p,q} \to G_{p,r}$ with $\Phi(A_{p,q}) = A_{p,r}$, i.e.\ the groups $G_{p,q}$ are rigid enough to have any isomorphism within this class also descend to the base groups $A_{p,q}$. 

\subsection{Reconstructing the associated subgroups from an isomorphism}\label{Subsec:reconstructing-u-v}

We now turn to the second reconstruction result: the above Lemma~\ref{Lem:iso-descends-to-A} showed that any isomorphism $G_{p,q} \cong G_{p,r}$ forces the existence of an isomorphism restricting to an isomorphism $A_{p,q} \cong A_{p,r}$. We now show that any \textit{such} isomorphism places a great deal of control over the images of the elements $c_q$ and $a_q b_q c_q$, the generators of the associated subgroups. 

\begin{lemma}\label{Lem:reconstructing-u-and-v}
Let $\Phi \colon G_{p,q} \to G_{p,r}$ be an isomorphism such that $\Phi(A_{p,q}) = A_{p,r}$. Then there exist $\alpha, \beta \in A_{p,r}$ and $\delta = \pm 1$ such that either
\begin{alignat*}{2}
\Phi(c_q) &= \alpha c_r^\delta \alpha^{-1} \quad &&\text{and} \quad \Phi(a_qb_qc_q) = \beta (a_rb_rc_r)^\delta \beta^{-1}; \\ 
\text{or else} \qquad \Phi(c_q) &= \alpha (a_rb_rc_r)^\delta \alpha^{-1} \quad &&\text{and} \quad \Phi(a_q b_q c_q) = \beta c_r^\delta \beta^{-1}.
\end{alignat*}
\end{lemma} 
\begin{proof}
Let $x \in \{ c_q, a_q b_q c_q \}$ and set $w = \Phi(x)$. Since $x \in A_{p,q}$, we have $w \in A_{p,r}$ by assumption. Since $\Phi$ is an isomorphism, it follows from Lemma~\ref{Lem:Centralizers-in-G}(ii) that $C_{p,r}(w) \cong F_2 \times \Z$, and in particular $\ZZ(C_{p,r}(w)) = \langle w \rangle$. Using this, we will prove that $w$ fixes an edge in its action on the Bass--Serre tree $\cT_r$ of $G_{p,r}$. Suppose it does not, in which case it fixes a unique vertex, which must be $x_0$ (whose stabilizer is $A_{p,r}$). Hence every element commuting with $w$ fixes $x_0$, so $C_{p,r}(w) \leq A_{p,r}$. If $w \not\in \ZZ(A_{p,r}) = \langle h_r \rangle$, then by Lemma~\ref{Lem:Centralizers-in-A}(i) we have that $C_{p,r}(w)$ is abelian; a contradiction, since it is isomorphic to $F_2 \times \Z$. On the other hand, if $w \in \ZZ(A_{p,r}) = \langle h_r \rangle$, then by Lemma~\ref{Lem:Centralizers-in-G}(i) we have $C_{p,r}(w) = A_{p,r}$. But $A_{p,r} \not\cong F_2 \times \Z$, since $A_{p,r}$ mod its center is $B_p \not\cong F_2$. Thus, $w$ must fix an edge. 

Consider the subtree $\Fix_{\cT_r}(w)$. This is connected and contains $x_0$, so $w$ fixes some edge incident to $x_0$. The stabilizers of these edges are $\alpha\langle c_r \rangle \alpha^{-1}$ and $\alpha \langle a_r b_r c_r \rangle \alpha^{-1}$ where $\alpha \in A_{p,r}$. Hence $w = \alpha y^k \alpha^{-1}$ for some $\alpha \in A_{p,r}$, some $y \in \{ c_r, a_rb_rc_r \}$ and some non-zero $k \in \Z$. By Lemma~\ref{Lem:Centralizers-in-G}(ii), the center of the centralizer of $y^k$ in $G_{p,r}$ is generated by $y$ (the $\Z$-factor in $F_2 \times \Z$). Hence $\ZZ (C_{p,r}(w)) = \alpha \langle y \rangle \alpha^{-1}$. But as concluded above, we also have $\ZZ(C_{p,r}(w)) = \langle w \rangle = \alpha \langle y^k \rangle \alpha^{-1}$. Hence $\langle y \rangle = \langle y^k \rangle$, forcing $k= \pm 1$, since $y$ has infinite order. 

We are almost done; applying the above argument to both cases of $x \in \{ c_q, a_qb_qc_q\}$, we have
\begin{equation}\label{Eq:phi-map-ambiguous}
\Phi(c_q) = \alpha y^\delta \alpha^{-1}, \quad \text{and} \quad \Phi(a_q b_q c_q) = \beta (y')^{\varepsilon} \beta^{-1}
\end{equation}
for some $\alpha, \beta \in A_{p,r}$ and $y, y' \in \{ c_r, a_rb_rc_r \}$, and $\delta, \varepsilon \in \{ -1, 1 \}$. We now do a final clean-up to reduce this ambiguity and obtain the statement of the lemma. The elements $c_q$ and $a_q b_q c_q$ are conjugate in $G_{p,q}$, so their images are conjugate in $G_{p,r}$, and hence the images get mapped to the same element under $\eta \colon G_{p,r} \to \Z$. Recall that $\eta(c_r) = \eta(a_rb_rc_r) = 1$, so $\delta = \eta(\Phi(c_q)) = \eta(\Phi(a_qb_qc_q)) = \varepsilon$. Finally, $y$ and $y'$ cannot be equal, as otherwise $\Phi(c_q)$ and $\Phi(a_qb_qc_q)$ would be conjugate inside $A_{p,r}$, so $c_q$ and $a_qb_qc_q$ would be conjugate inside $A_{p,q}$. But then modulo the center $\langle h_q \rangle$, they would remain conjugate inside $B_p$, which is impossible, since we observed that $\langle c_q \rangle$ and $\langle a_q b_q c_q \rangle$ are not conjugate inside $B_p$ at the very beginning of \S\ref{Subsec:rarely-isomorphic}. Hence $\{ y, y' \} = \{ c_r, a_rb_rc_r \}$, and since $\varepsilon=\delta$, we have that \eqref{Eq:phi-map-ambiguous} is precisely the statement of the lemma. 
\end{proof}

We now have all of the ingredients required to prove the main proposition of this section: any isomorphism $G_{p,q} \cong G_{p,r}$ forces an isomorphism that maps the base group $A_{p,q}$ to $A_{p,r}$, and which is very rigid on the associated subgroups $\langle c_q \rangle$ and $\langle a_qb_qc_q \rangle$. 

\subsection{Proof of Proposition~\ref{Prop:rarely-isomorphic}}\label{Subsec:Proof-of-rarely-iso}

The reverse direction is very quick: if $r \equiv \pm q \pmod{p}$, then it is not difficult to show that the same map as defined in \eqref{Eq:im-of-h_q-generators} defines a homomorphism $G_{p,q} \to G_{p,r}$, and that we can choose $t = \pm 1$; the analogous map $G_{p,r} \to G_{p,q}$ can then easily be seen to be an inverse, and so $G_{p,q} \cong G_{p,r}$. Since we will not make any use of this direction of the proposition, we leave the details of verifying this to the reader.

We now prove the crucial forward direction. Suppose that $G_{p,q} \cong G_{p,r}$. Then by Lemma~\ref{Lem:iso-descends-to-A}, and the remarks following it, there exists an isomorphism $\Phi \colon G_{p,q} \to G_{p,r}$ such that $\Phi(A_{p,q}) = A_{p,r}$. Fix such an isomorphism $\Phi$. Since $\ZZ(A_{p,q}) = \langle h_q \rangle$ and $\ZZ(A_{p,r}) = \langle h_r \rangle$, it follows that $\Phi(h_q) = h_r^\varepsilon$ for some $\varepsilon = \pm 1$. Passing to the quotients by these centers, we have an automorphism $\overline{\Phi} \colon B_p \to B_p$. Since the two free factors $\langle a_r \rangle$ and $\langle b_r \rangle$ represent the two conjugacy classes of maximal finite subgroups of $B_p$, the automorphism $\overline{\Phi}$ must permute these two conjugacy classes. After interchanging $a_r$ and $b_r$ if necessary, there exist $\mu, \nu \in (\Z / p\Z)^\times$ such that in the abelianization $B_p^\ab = \Z /p\Z \oplus \Z / p\Z \oplus \Z$ (with basis ordered as $[a_q], [b_q], [c_q]$, where $[ \cdot ]$ denotes the class in the abelianization) we have that the map induced by $\overline{\Phi}$ sends $[a_q] \mapsto \mu [a_r]$ and $[b_q] \mapsto \nu [b_r]$. 

We claim that $\mu, \nu \equiv \pm 1 \pmod{p}$. There are two cases to consider, corresponding to the two cases in Lemma~\ref{Lem:reconstructing-u-and-v}. Let $\alpha, \beta \in A_{p,r}$ and $\delta = \pm 1$ be as in the statement of the lemma, and suppose that the first case holds. Letting $\sim$ denote conjugacy, the first case holding implies then that $\overline{\Phi}(c_q) \sim c_r^\delta$ and $\overline{\Phi}(a_q b_q c_q) \sim (a_r b_r c_r)^\delta$. Hence $[\overline{\Phi}(c_q)] =\delta [c_r]$, so
\[
\mu[a_r] + \nu [b_r] + \delta [c_r] = \delta [a_r] + \delta[b_r] + \delta [c_r].
\]
Comparing coefficients, in this first case we thus have $\mu \equiv \nu \equiv \delta \pmod{p}$, since $p[a_r] = p[b_r] = 0$. In particular, since $\delta = \pm 1$, we have $\mu, \nu \equiv \pm 1 \pmod{p}$.

If instead the second case of Lemma~\ref{Lem:reconstructing-u-and-v} holds, then very similarly we get that $\overline{\Phi}(c_q) \sim (a_r b_r c_r)^\delta$ and $\overline{\Phi}(a_q b_q c_q) \sim c_r^\delta$, so that $[\overline{\Phi}(c_q)] = \delta ([a_r] + [b_r] + [c_r])$. Comparing this image with $\delta[c_r]$ gives 
\[
\delta [c_r] = (\mu + \delta)[a_r] + (\nu + \delta)[b_r] + \delta[c_r]
\]
implying that in this second case we have $\mu \equiv \nu \equiv -\delta \pmod{p}$. Thus in this case we also have $\mu, \nu \equiv \pm 1 \pmod{p}$, since $\delta \in \{-1, 1\}$, completing the proof of our claim. 

We are now almost done. The element $\overline{\Phi}(a_q)$ is conjugate in $B_p$ to a power $x^\mu$, where $x$ is either $a_r$ or $b_r$. Lifting the conjugating element to $A_{p,r}$, we may thus write 
\begin{equation}\label{Eq:image-of-a_q}
\Phi(a_q) = g x^\mu g^{-1} h_r^j
\end{equation}
for some $g \in A_{p,r}$ and $j \in \Z$. We have a defining relation $x^p = h_r^{-\rho}$ in $A_{p,r}$, where $\rho=r$ if $x=a_r$ and $\rho = -r$ if $x = b_r$. Taking $p$th powers in \eqref{Eq:image-of-a_q}, and using $\Phi(h_q) = h_r^{\varepsilon}$, we thus get 
\[
h_r^{-\varepsilon q} = \Phi(a_q^p) = \Phi(a_q)^p = g x^{\mu p}g^{-1} h_r^{jp} = h_r^{-\rho \mu + jp}.
\]
Since $h_r$ has infinite order, we have $-\varepsilon q = -\rho \mu + jp$, so that $\varepsilon q \equiv \rho \mu \pmod{p}$. But $\mu \equiv \pm 1 \pmod{p}$, so the right-hand side is congruent to $\pm r$ mod $p$. Since $\varepsilon = \pm 1$, we thus have 
\[
q \equiv \pm r \pmod{p}.
\]
This is what was to be shown, and this completes the proof of Proposition~\ref{Prop:rarely-isomorphic}. \qed

\bibliographystyle{amsalpha}
\bibliography{freebycyclic-pr.bib}

@article {Andrew2025,
    AUTHOR = {Andrew, Naomi and Hillen, Paige and Lyman, Robert Alonzo and
              Pfaff, Catherine Eva},
     TITLE = {A hyperbolic free-by-cyclic group determined by its finite
              quotients},
   JOURNAL = {Glasg. Math. J.},
  FJOURNAL = {Glasgow Mathematical Journal},
    VOLUME = {67},
      YEAR = {2025},
    NUMBER = {3},
     PAGES = {500--502},
      ISSN = {0017-0895,1469-509X},
   MRCLASS = {20E26 (20E05 20E36 20F65)},
  MRNUMBER = {4941614},
MRREVIEWER = {Marco\ Trombetti},
       DOI = {10.1017/S0017089525000096},
       URL = {https://doi.org/10.1017/S0017089525000096},
}

@article {Baumslag1971,
    AUTHOR = {Baumslag, Gilbert},
     TITLE = {Finitely generated cyclic extensions of free groups are
              residually finite},
   JOURNAL = {Bull. Austral. Math. Soc.},
  FJOURNAL = {Bulletin of the Australian Mathematical Society},
    VOLUME = {5},
      YEAR = {1971},
     PAGES = {87--94},
      ISSN = {0004-9727},
   MRCLASS = {20E25},
  MRNUMBER = {311776},
MRREVIEWER = {C.\ F.\ Miller, III},
       DOI = {10.1017/S0004972700046906},
       URL = {https://doi.org/10.1017/S0004972700046906},
}

@article {BMRS,
    AUTHOR = {Bridson, M. R. and McReynolds, D. B. and Reid, A. W. and
              Spitler, R.},
     TITLE = {Absolute profinite rigidity and hyperbolic geometry},
   JOURNAL = {Ann. of Math. (2)},
  FJOURNAL = {Annals of Mathematics. Second Series},
    VOLUME = {192},
      YEAR = {2020},
    NUMBER = {3},
     PAGES = {679--719},
      ISSN = {0003-486X,1939-8980},
   MRCLASS = {57M50 (11F06 20E18 20H10)},
  MRNUMBER = {4172619},
       DOI = {10.4007/annals.2020.192.3.1},
       URL = {https://doi.org/10.4007/annals.2020.192.3.1},
}

@article {BMRS21,
    AUTHOR = {Bridson, Martin R. and McReynolds, D. B. and Reid, Alan W. and
              Spitler, Ryan},
     TITLE = {On the profinite rigidity of triangle groups},
   JOURNAL = {Bull. Lond. Math. Soc.},
  FJOURNAL = {Bulletin of the London Mathematical Society},
    VOLUME = {53},
      YEAR = {2021},
    NUMBER = {6},
     PAGES = {1849--1862},
      ISSN = {0024-6093,1469-2120},
   MRCLASS = {20H10 (11F06)},
  MRNUMBER = {4386043},
MRREVIEWER = {Alexander\ W.\ Mason},
       DOI = {10.1112/blms.12546},
       URL = {https://doi.org/10.1112/blms.12546},
}

@article {Bridson2017,
    AUTHOR = {Bridson, Martin R. and Reid, Alan W. and Wilton, Henry},
     TITLE = {Profinite rigidity and surface bundles over the circle},
   JOURNAL = {Bull. Lond. Math. Soc.},
  FJOURNAL = {Bulletin of the London Mathematical Society},
    VOLUME = {49},
      YEAR = {2017},
    NUMBER = {5},
     PAGES = {831--841},
      ISSN = {0024-6093,1469-2120},
   MRCLASS = {57M27 (20E18 20E26)},
  MRNUMBER = {3742450},
MRREVIEWER = {Steffen\ Kionke},
       DOI = {10.1112/blms.12076},
       URL = {https://doi.org/10.1112/blms.12076},
}

@incollection {Bridson2020,
    AUTHOR = {Bridson, Martin R. and Reid, Alan W.},
     TITLE = {Profinite rigidity, fibering, and the figure-eight knot},
 BOOKTITLE = {What's next?---the mathematical legacy of {W}illiam {P}.
              {T}hurston},
    SERIES = {Ann. of Math. Stud.},
    VOLUME = {205},
     PAGES = {45--64},
 PUBLISHER = {Princeton Univ. Press, Princeton, NJ},
      YEAR = {2020},
      ISBN = {978-0-691-18589-7; 978-0-691-16776-3; 978-0-691-16777-0},
   MRCLASS = {57K30 (20E18 57K10)},
  MRNUMBER = {4205635},
       DOI = {10.2307/j.ctvthhdvv.6},
       URL = {https://doi.org/10.2307/j.ctvthhdvv.6},
}

@article {Bridson2022,
    AUTHOR = {Bridson, M. R. and Reid, A. W.},
     TITLE = {Profinite rigidity, {K}leinian groups, and the cofinite {H}opf
              property},
   JOURNAL = {Michigan Math. J.},
  FJOURNAL = {Michigan Mathematical Journal},
    VOLUME = {72},
      YEAR = {2022},
     PAGES = {25--49},
      ISSN = {0026-2285,1945-2365},
   MRCLASS = {20H10 (20E18 22E40 30F40 57M50)},
  MRNUMBER = {4460248},
MRREVIEWER = {Alexander\ W.\ Mason},
       DOI = {10.1307/mmj/20217218},
       URL = {https://doi.org/10.1307/mmj/20217218},
}

@article {Bridson2025,
    AUTHOR = {Bridson, Martin R. and Piwek, Pawe{\l}},
     TITLE = {Profinite rigidity for free-by-cyclic groups with centre},
   JOURNAL = {J. Lond. Math. Soc. (2)},
  FJOURNAL = {Journal of the London Mathematical Society. Second Series},
    VOLUME = {111},
      YEAR = {2025},
    NUMBER = {6},
     PAGES = {Paper No. e70181, 28},
      ISSN = {0024-6107,1469-7750},
   MRCLASS = {20E18 (20F65)},
  MRNUMBER = {4920376},
MRREVIEWER = {J.\ W.\ MacQuarrie},
       DOI = {10.1112/jlms.70181},
       URL = {https://doi.org/10.1112/jlms.70181},
}

@article {Dixon1982,
    AUTHOR = {Dixon, John D. and Formanek, Edward W. and Poland, John C. and
              Ribes, Luis},
     TITLE = {Profinite completions and isomorphic finite quotients},
   JOURNAL = {J. Pure Appl. Algebra},
  FJOURNAL = {Journal of Pure and Applied Algebra},
    VOLUME = {23},
      YEAR = {1982},
    NUMBER = {3},
     PAGES = {227--231},
      ISSN = {0022-4049,1873-1376},
   MRCLASS = {20E18},
  MRNUMBER = {644274},
MRREVIEWER = {S.\ P.\ Demushkin},
       DOI = {10.1016/0022-4049(82)90098-6},
       URL = {https://doi.org/10.1016/0022-4049(82)90098-6},
}

@article {Grunewald2011,
    AUTHOR = {Grunewald, Fritz and Zalesskii, Pavel},
     TITLE = {Genus for groups},
   JOURNAL = {J. Algebra},
  FJOURNAL = {Journal of Algebra},
    VOLUME = {326},
      YEAR = {2011},
     PAGES = {130--168},
      ISSN = {0021-8693,1090-266X},
   MRCLASS = {20E18 (20E26)},
  MRNUMBER = {2746056},
       DOI = {10.1016/j.jalgebra.2010.05.018},
       URL = {https://doi.org/10.1016/j.jalgebra.2010.05.018},
}

@article {Hughes2025,
    AUTHOR = {Hughes, Sam and Kudlinska, Monika},
     TITLE = {On profinite rigidity amongst free-by-cyclic groups {I}: {T}he
              generic case},
   JOURNAL = {Proc. Lond. Math. Soc. (3)},
  FJOURNAL = {Proceedings of the London Mathematical Society. Third Series},
    VOLUME = {130},
      YEAR = {2025},
    NUMBER = {6},
     PAGES = {Paper No. e70059, 43},
      ISSN = {0024-6115,1460-244X},
   MRCLASS = {20E18 (20E26 20E36 20F65 20F67 20J05 20J06)},
  MRNUMBER = {4920397},
MRREVIEWER = {Benjamin\ Klopsch},
       DOI = {10.1112/plms.70059},
       URL = {https://doi.org/10.1112/plms.70059},
}

@article {Linton2026,
    AUTHOR = {Linton, Marco},
     TITLE = {Embedding finitely generated free-by-cyclic groups in
              \{finitely generated free\}-by-cyclic groups},
   JOURNAL = {Int. Math. Res. Not. IMRN},
  FJOURNAL = {International Mathematics Research Notices. IMRN},
      YEAR = {2026},
    NUMBER = {4},
     PAGES = {Paper No. rnag020, 15},
      ISSN = {1073-7928,1687-0247},
   MRCLASS = {20E05 (20E06 20F65)},
  MRNUMBER = {5035094},
       DOI = {10.1093/imrn/rnag020},
       URL = {https://doi.org/10.1093/imrn/rnag020},
}

@incollection {Noskov1979,
    AUTHOR = {Noskov, G. A. and Remeslennikov, V. N. and Roman'kov,
              V. A.},
     TITLE = {Infinite groups},
 BOOKTITLE = {Algebra. {T}opology. {G}eometry, {V}ol. 17 ({R}ussian)},
    SERIES = {Itogi Nauki i Tekhniki},
     PAGES = {65--157, 308},
 PUBLISHER = {Akad. Nauk SSSR, Vsesoyuz. Inst. Nauchn. i Tekhn. Inform.,
              Moscow},
      YEAR = {1979},
   MRCLASS = {20-02},
  MRNUMBER = {584569},
MRREVIEWER = {Yu.\ I.\ Merzlyakov},
}

@inproceedings {Reid2018,
    AUTHOR = {Reid, Alan W.},
     TITLE = {Profinite rigidity},
 BOOKTITLE = {Proceedings of the {I}nternational {C}ongress of
              {M}athematicians---{R}io de {J}aneiro 2018. {V}ol. {II}.
              {I}nvited lectures},
     PAGES = {1193--1216},
 PUBLISHER = {World Sci. Publ., Hackensack, NJ},
      YEAR = {2018},
      ISBN = {978-981-3272-91-0; 978-981-3272-87-3},
   MRCLASS = {20E18 (57M07)},
  MRNUMBER = {3966805},
MRREVIEWER = {Gareth\ Wilkes},
}

@article {Nikolov2007,
    AUTHOR = {Nikolov, Nikolay and Segal, Dan},
     TITLE = {On finitely generated profinite groups. {I}. {S}trong
              completeness and uniform bounds},
   JOURNAL = {Ann. of Math. (2)},
  FJOURNAL = {Annals of Mathematics. Second Series},
    VOLUME = {165},
      YEAR = {2007},
    NUMBER = {1},
     PAGES = {171--238},
      ISSN = {0003-486X,1939-8980},
   MRCLASS = {20E18 (20E32 20F12)},
  MRNUMBER = {2276769},
MRREVIEWER = {Benjamin\ Klopsch},
       DOI = {10.4007/annals.2007.165.171},
       URL = {https://doi.org/10.4007/annals.2007.165.171},
}

@book {Serre,
    AUTHOR = {Serre, Jean-Pierre},
     TITLE = {Trees},
      NOTE = {Translated from the French by John Stillwell},
 PUBLISHER = {Springer-Verlag, Berlin-New York},
      YEAR = {1980},
     PAGES = {ix+142},
      ISBN = {3-540-10103-9},
   MRCLASS = {20H10 (05C05 22E50)},
  MRNUMBER = {607504},
}

\end{document}